\documentclass{amsproc}

\usepackage{amsmath,amsfonts,amscd,amssymb,amsthm,epsfig,euscript,eufrak}
\usepackage{mathrsfs}
\usepackage{amsxtra}
\usepackage[all]{xy}
\usepackage{verbatim,epsf}
\newtheorem{theorem}{Theorem}

\newtheorem{lemma}{Lemma}
\newtheorem{corollary}{Corollary}
\theoremstyle{definition}
\newtheorem{definition}{Definition}

\theoremstyle{remark}
\newtheorem{remark}{Remark}
\numberwithin{equation}{section}
\newcommand{\field}[1]{\ensuremath{\mathbb{#1}}}
\newcommand{\CC}{\field{C}}
\newcommand{\DD}{\field{D}}
\newcommand{\HH}{\field{H}}

\newcommand{\RR}{\field{R}}

\DeclareMathOperator{\re}{Re}

\newcommand{\delb}{\bar\partial}

\newcommand{\fS}{\mathfrak{S}}

\newcommand{\curly}[1]{\mathscr{#1}}
\newcommand{\cA}{\curly{A}}

\newcommand{\cC}{\curly{C}}

\newcommand{\cE}{\curly{E}}
\newcommand{\cF}{\curly{F}}

\newcommand{\cH}{\curly{H}}

\newcommand{\cK}{\curly{K}}
\newcommand{\cL}{\curly{L}}

\newcommand{\cN}{\curly{N}}

\newcommand{\cP}{\curly{P}}

\newcommand{\Ga}{\Gamma}

\newcommand{\PSL}{\mathrm{PSL}(2,\mathbb{R})}

\newcommand{\Ur}{\mathrm{U}(r)}

\newcommand{\OO}{\mathcal{O}}

 \DeclareMathOperator{\Hom}{Hom}
 
\DeclareMathOperator{\Aut}{Aut}
 
\DeclareMathOperator{\GL}{GL} 

 \DeclareMathOperator{\Ad}{Ad}

\usepackage{hyperref}
\hypersetup{colorlinks,citecolor=blue,plainpages=false,hypertexnames=false}

\numberwithin{equation}{section}

\begin{document}

\title[On Shimura's isomorphism]{On Shimura's isomorphism and $(\Ga, G)$-bundles on the upper-half plane}


\author[Meneses]{Claudio Meneses}
\address{Max Planck Institut f\"ur Mathematik, Vivatsgasse 7, 53111 Bonn, Germany}
\curraddr{CIMAT A.C., 
  Jalisco S/N, Col. Valenciana CP: 36240, Gto., Mexico}
\email{claudio.meneses@cimat.mx, claudio.meneses.torres@gmail.com}
\thanks{}


\subjclass[2010]{Primary 30F35, 32G13; Secondary 20J06}

\date{}

\begin{abstract}
For a compact real form $U$ of a complex simple Lie group $G$, and an irreducible representation $\rho:\Ga \to U$ of a Fuchsian group of the first kind $\Ga$, it is shown that the classical isomorphism of Shimura, for the periods of a cusp form of weight 2 with values in $\mathfrak{g}$ and the representation $\Ad\rho:\Ga\to\Aut\mathfrak{g}$, can be interpreted as the differential at a point of the zero section, for a natural map from the cotangent bundle of the moduli space of certain $(\Gamma, G)$-bundles over $\HH$ (in the sense of Seshadri) to an open set in the smooth locus of the character variety $\Hom_{\mathbf{g}}(\Ga,G)/PG$. Emphasis is put on analytic techniques.
\end{abstract}

\maketitle


\tableofcontents
\section{Introduction}\label{section:intro}
Let $G$ be a complex finite-dimensional simple Lie group with Lie algebra $\mathfrak{g}$, $U\subset G$ a compact real form, and $\rho:\Gamma\to U$ an irreducible representation of an arbitrary Fuchsian group of the first kind $\Ga$.   A \emph{$\mathfrak{g}$-valued cusp form of weight 2 with the representation $\Ad\rho$} is a holomorphic function $\Phi:\HH\to \mathfrak{g}$ satisfying the functional equation
\begin{equation}\label{automorphic}
\Phi(\gamma\tau)\gamma'(\tau)=\Ad\rho(\gamma)\cdot \Phi(\tau),\qquad \forall\gamma\in\Ga,\;\tau\in\HH,
\end{equation}
where $\Ad\rho$ is the composition of $\rho$ and the adjoint representation of $G$ in $\mathfrak{g}$,
such that for every cusp $\tau_{i}\in\RR\cup\{\infty\}$, 
\begin{equation}\label{limit}
\lim_{\tau\to\tau_{i}}\Phi(\tau)=0,
\end{equation}
where the limit is understood to be taken within a fundamental region $\mathfrak{F}$. The spaces $\mathfrak{S}_{2}(\Ga,\Ad\rho)$ of cusp forms of weight 2 (and of arbitrary positive even weight) were introduced by Shimura \cite{Shimura71}, are finite dimensional, and thanks to the Killing form in $\mathfrak{g}$, admit a Hermitian inner product generalizing the Petersson inner product \cite{Petersson40, Petersson48}. 

The author's interest in the subject stems from the humble realization that the theory of moduli of vector bundles with parabolic structures over Riemann surfaces, as in the work of  Narasimhan-Mehta-Seshadri \cite{N-S65, MS80},  can be fully understood in terms of complex analysis on the upper half-plane, and more concretely, the harmonic theory for families of Cauchy-Riemann operators. Such a complex-analytic approach, already implicit in the seminal work of Andr\'e Weil \cite{Weil38}, has been pursued, for instance, in \cite{TZ07}, 
and is sufficiently flexible to allow the replacement of the groups $\Ur\subset\GL(r,\CC)$,  by general simple pairs $U\subset G$, and their corresponding geometric structures. 
These are the objects we will concentrate our attention on in this paper. 

Ideally, one wishes to associate, to every irreducible homomorphism $\rho:\Ga\to U$, some suitable geometric data consisting of 
\begin{itemize}
\item[(i)] A holomorphic principal $G$-bundle $P$ over the compact Riemann surface $\Ga\setminus\HH^{*}$, with additional data in the fibers of each cusp and elliptic fixed point, generalizing the notion of a parabolic structure, and satisfying a suitable notion of parabolic stability,
\item[(ii)] A reduction of the structure group of $P$ to $U$, away from the cusps and elliptic fixed points,
\item[(iii)] An irreducible singular flat $U$-connection on $P$, whose holonomy representation corresponds to $\rho$. 
\end{itemize}
One would then expect a correspondence between the set of $U$-conjugacy clases $\{\rho\}$ of irreducible representations $\rho$ for which each parabolic and elliptic generator lies in an a priori given conjugacy class in $U$, and the set of equivalence classes $\{P\}$ of stable parabolic $G$-bundles on $\Ga\setminus\HH^{*}$ with prescribed weight data, explicitly realized by means of the existence of a singular irreducible flat $U$-connection in $P$. While this is true for stable principal $G$-bundles \cite{Ramanathan75} (i.e., purely hyperbolic Fuchsian groups), it turns out that such association is not possible in general for arbitrary Fuchsian groups of the first kind, and only in the generic case when the weights (for all parabolic and elliptic generators) lie in the interior of a Weyl alcove for a choice of maximal torus in $U$. In general, the geometric structures suitable for this purpose have been called \emph{parahoric G-torsors} by Balaji-Seshadri (see \cite{BS15} and references therein), and indeed satisfy the previous uniformization principle, generalizing the results in \cite{N-S65,MS80}. Due two the concrete objectives of the present work, our perspective is based, instead, on working over $\HH$ in terms of a complex-analytic framework, and directly considering some induced spaces of global holomorphic sections (in the sense of sheaf cohomology) with suitable asymptotic behavior, as spaces of $\mathfrak{g}$-valued automorphic forms satisfying the cusp form condition. 

In another stream of ideas, the study of arithmetic properties of automorphic functions let Goro Shimura to generalize the work of Eichler to periods associated to representations of Fuchsian groups on compact groups. For weight 2, the celebrated isomorphism of Shimura \cite{Shimura71} reads
\[
\mathfrak{S}_{2}(\Ga,\Ad\rho)\oplus\overline{\mathfrak{S}_{2}(\Ga,\Ad\rho)}\cong H_{P}^{1}(\Ga,\mathfrak{g}_{\Ad\rho}).
\]
One is undoubtedly left to notice that, on one hand, the parabolic cohomology $H_{P}^{1}(\Ga,\mathfrak{g}_{\Ad\rho})$ models the tangent space at a $G$-equivalence class $\{\rho\}$ of representations in a complex character variety $\cK_{\CC}$, and on the the other, the spaces $\mathfrak{S}_{2}(\Ga,\Ad\rho)$ and $\overline{\mathfrak{S}_{2}(\Ga,\Ad\rho)}$ model, respectively, the tangent and cotangent spaces at a point in a certain moduli space $\cN$ of $(\Ga,G)$-principal bundles over $\HH$ \cite{TZ07} (cf. \cite{Ivanov88}). 

The two previous pictures get connected after one notices it is possible to associate, in a natural and injective way, an irreducible complex representation $\chi:\Ga\to G$ to a pair consisting of an irreducible representation $\rho:\Ga\to U$, and an element $\Phi\in\mathfrak{S}_{2}(\Gamma,\Ad\rho)$ (see section \ref{section:deformation}), in such a way that the $(\Ga,G)$-bundles over $\HH$ induced by $\chi$ and $\rho$ are holomorphically equivalent, although equipped with nonequivalent singular flat connections.\footnote{A word of caution: $\chi$ \emph{is not the same} than the induced representation considered in the study of harmonic metrics when $G=\GL(r,\CC)$.} We refer to such correspondence as the \emph{Narasimhan-Mehta-Seshadri map}. Considering the zero section $s_{0}$ of the cotangent bundle $T^{*}\cN$, induces a canonical splitting 
\[
s_{0}^{*}\left(T(T^{*}\cN) \right)\cong T^{*}\cN\oplus T\cN,
\]
as vector bundles over $\cN$. Thus, Shimura's work on parabolic cohomology \cite{Shimura71} admits a natural geometric interpretation in the previous setup (a fact observed by Goldman \cite{Gol84} for the group $\PSL$, scalar cusp forms of higher weight, and their Eichler periods). The main result of this work is the following theorem.

\begin{theorem}\label{main theo}
With respect to the complex coordinates given by the choice of an arbitrary basis in $\mathfrak{S}_{2}(\Gamma,\Ad\rho)\oplus\overline{\mathfrak{S}_{2}(\Gamma,\Ad\rho)}$, the differential of the Narasimhan-Mehta-Seshadri map $\mathscr{F}:T^{*}\cN^{s}\to\cK_{\CC}$ at the point $\left\{P_{\rho},0\right\}\in T^{*}\cN^{s}$ in the image of the zero section  $s_{0}: \cN^{s} \to T^{*}\cN^{s}$ is given as 
\begin{equation}
d\cF_{\{P_{\rho},0\}}(\Psi,\nu)= \cL_{S}\left(\Psi,\nu\right),
\end{equation}
where $\cL_{S}$ is Shimura's isomorphism.
\end{theorem}

In other words, theorem \ref{main theo} states that Shimura's isomorphism \cite{Shimura71} between complex parabolic cohomology and spaces of $\mathfrak{g}$-valued cusp forms of weight 2 with the representation $\Ad\rho$ is essentially the infinitesimal version of the Narasimhan-Mehta-Seshadri map $\cF:T^{*}\cN^{s}\to\cK_{\CC}$ over $s_{0}$, a fact that illustrates explicitly the trascendental nature of $\cF$. 
In fact, it is possible to say more: $\cF$ is a real-analytic monomorphism, and $\cF \circ s_{0}$ is also a symplectomorphism onto its image, once the unitary character variety $\cK^{s}$ is endowed with Goldman's symplectic form, and the natural K\"ahler structure induced from Petersson inner product is considered over $\cN^{s}$. 


The contents of this work are organized as follows: section \ref{section:generalities} presents the conventions in which the present work is founded; section \ref{section:char} describes the construction of the character variety $\Hom_{\mathbf{t}}(\Ga,G)/G$; in section \ref{section:deformation}, we develop the infinitesimal deformation theory of $(\Ga, G)$-bundles from a complex analytic perspective. Finally, a proof of theorem \ref{main theo} (and some corollaries) is presented in section \ref{section:Shimura}. 

I am deeply grateful to the Max Planck Institute for Mathematics in Bonn, where the present work was developed, for its generosity and stimulating working atmosphere, and to CIMAT in Mexico, where the manuscript was put in final form. I would also like to warmly thank Indranil Biswas for pointing out an important issue in the first version of this manuscript, and the referee for the useful comments provided.

\section{Lie algebra-valued cusp forms of weight 2}\label{section:generalities}

Let $\Ga$ be a Fuchsian group (i.e. a discrete subgroup of $\PSL$)  of the first kind. 
It is a classical result \cite{Hejhal83,Lehner64}
that $\Gamma$ is finitely generated and admits an explicit presentation in terms of $2g$ hyperbolic generators $A_{1},\dots,A_{g},B_{1},\dots,B_{g},$ $m$ elliptic generators $S_{1},\dots,S_{m}$, and $n$ parabolic generators $T_{m+1},\cdots,T_{m+n}$, with relations 
\[
\prod_{i=1}^{g}[A_{i},B_{i}]\cdot\prod_{j=1}^{m}S_{j}\cdot\prod_{k=1}^{n}T_{m+k}=e,  
\]
\[
S_{1}^{c_{1}}=\cdots =S_{m}^{c_{m}}=e,
\]
for some integers $2\leq c_{j}<\infty$. When talking about a Fuchsian group, we will assume to have fixed once and for all one such \emph{marking}.
 
The \emph{signature} of $\Gamma$ is the collection of nonnegative integers $(g;c_{1},\dots,c_{m};n)$, $j=1,\dots,m$, subject to the hyperbolicity condition 
\[
2g-2 +\sum_{i=1}^{m}\left(1-\frac{1}{c_{i}}\right)+n > 0.
\]
The equivalence relation on Fuchsian groups is conjugation in $\PSL$.
$\Gamma$ admits a fundamental region $\mathfrak{F}$ of finite hyperbolic area, bounded by a finite number of geodesic arcs and vertices determined by the fixed points of the generators $\{S_{1},\dots,S_{m},T_{m+1},\dots,T_{m+n}\}$. We will denote by $\{e_{1},\cdots,e_{m}\}\subset\mathbb{H}$ the set of elliptic fixed points ($S_{i}(e_{i})=e_{i}$), and by $\{p_{1},\cdots,p_{n}\}\subset\mathbb{R}\cup\{\infty\}$ the set of cusps ($T_{m+i}(p_{i})=p_{i}$) in the fundamental region $\mathfrak{F}$. We follow Shimura's convention \cite{Shimura71} on the construction of the compact quotient 
\[
S=\Gamma\setminus\mathbb{H}^{*}
\]
with local elliptic (resp. parabolic) coordinates  $\zeta_{i}=\zeta^{c_{i}}\circ\sigma^{-1}_{i}$ (resp. $q_{i}=q\circ\sigma_{i}^{-1}$) where $q=e^{2\pi\sqrt{-1}\tau}$, $\zeta:\HH\to\DD$ is the Cayley mapping, and $\sigma_{1},\dots,\sigma_{m+n} \in \text{PSL}(2,\mathbb{R})$ satisfy 
\begin{align}
\zeta\left((\sigma_{i}^{-1}S_{i}\sigma_{i})(\tau)\right)
& = e^{2\pi\sqrt{-1}/c_{i}}\cdot\zeta(\tau), \qquad & i\leq m \label{can-elliptic}\\
\left(\sigma_{i}^{-1}T_{i}\sigma_{i}\right)(\tau) & =\tau\pm 1,\qquad & i>m. \label{can-cusp}
\end{align}

\subsection{Spaces of $\mathfrak{g}$-valued cusp forms} 

Let $\rho:\Ga\to U$ be irreducible with induced complex representation $\Ad\rho:\Ga\to\Aut\mathfrak{g}$, $p,q\in\{0,1\}$, and let us denote by $\cC^{p,q}_{c}(\Ga,\Ad\rho)$ the space of smooth, $\mathfrak{g}$-valued functions on $\HH$, compactly supported within a fundamental region $\mathfrak{F}$ for $\Ga$, that satisfy
\[
\Phi(\gamma\tau)\gamma'(\tau)^{p}\overline{\gamma'(\tau)^{q}}=\Ad\rho(\gamma)\cdot \Phi(\tau)\qquad \forall\;\tau\in\HH,\;\;\gamma\in\Ga.
\]
For any 
$\Phi_{1},\Phi_{2}\in\cC^{p,q}_{c}(\Ga,\Ad\rho)$, their \emph{Petersson inner product} is defined as 
\begin{equation}\label{Petersson}
\langle \Phi_{1},\Phi_{2}\rangle_{P} = \iint\limits_{\mathfrak{F}}\langle \Phi_{1},\Phi_{2}\rangle\;y^{2p+2q-2}dx dy
\end{equation}
where $\tau=x+\sqrt{-1}y$, and $\langle \Phi_{1},\Phi_{2}\rangle=-K\left(\Phi_{1},\overline{\Phi_{2}}\right)$ is, pointwise,  the Hermitian inner product in $\mathfrak{g}$ determined by the Killing form $K$ and the Cartan involution $X\mapsto \overline{X}$ defined by complex conjugation of the realization $\mathfrak{g}\cong\mathfrak{u}_{\CC}$ (for $\mathfrak{u}$ the Lie algebra of $U$) \cite{Hel79}. A fundamental property of the Petersson inner product is its invariance under the induced action of $\Ga$ in $\cC^{p,q}_{c}(\Ga,\Ad\rho)$. Let us denote by $\cH^{p,q}(\Gamma,\Ad\rho)$ the completion of $\cC^{p,q}_{c}(\Ga,\Ad\rho)$ to a Hilbert space under the Petersson inner product. In terms of the Cauchy-Riemann operators 
\[
\bar{\partial}^{p,q}:\cH^{p,q}(\Gamma,\Ad\rho)\to\cH^{p,q+1}(\Gamma,\Ad\rho),
\] 
it is also possible to define the spaces of $\mathfrak{g}$-valued cusp forms of weight 2 as
\[
\mathfrak{S}_{2}(\Ga,\Ad\rho)=\ker\bar{\partial}^{1,0}.
\]

\section{Character varieties for simple groups}\label{section:char}

We recall the construction of the smooth locus of a character variety for a Fuchsian group $\Ga$ and a simple group $G$. For details, the reader is referred to \cite{Gol84} and references therein. 

Let $T\subset U$ be a choice of maximal torus for $U\subset G$, and let $\mathfrak{t}\subset\mathfrak{u}$ denote its Lie algebra. Let us fix, for each parabolic and elliptic generator of $\Ga$, an element $g_{i}\in T$. Each such $g_{i}$ has an induced adjoint orbit in $G$, $\OO_{i}=[g_{i}]\subset G$, $i=1,\dots,m+n$. We will denote $\mathbf{g}=\{g_{i}\}_{i=1}^{m+n}$. Consider the subset $\Hom_{\mathbf{g}}(\Ga,G)\subset\Hom(\Ga,G)$ consisting of representations $\chi:\Ga\to G$ with $\chi(S_{i})\in\OO_{i}$, $\chi(T_{i})\in\OO_{i}$. The smooth locus $\cP$ in $\Hom_{\mathbf{g}}(\Ga,G)$ is the collection of representations $\chi:\Ga\to G$ satisfying 
\begin{equation*}
\dim H^{0}(\Ga,\Ad\chi) = \dim \mathfrak{g}^{\Ad\chi} = \dim \mathfrak{g}^{\Ad G}=0,
\end{equation*}
which is an open condition. More specifically, $\Hom_{\mathbf{g}}(\Ga,G)=\psi^{-1}(I)$, where $\psi$ is the map 
\[
\psi :G^{2g}\times
\prod_{i=1}^{n+m}\mathcal{O}^{i}
\rightarrow U,
\]
\[
(M_{1},N_{1},\dots,M_{g},N_{g},R_{1},\dots,R_{n+m})
\mapsto  \prod_{i=1}^{g}[M_{i},N_{i}]\prod_{i=1}^{n+m} R_{i},
\]
and the smooth locus is the subset $\cP\subset\psi^{-1}(I)$ of regular points of $\psi$, consisting of representations $\chi$ for which $\dim Z(\chi)=\dim Z(G)=0$.
Even though the action of $PG=G/Z(G)$ on $\cP$ by conjugation is locally free and generically free, it is not proper in general. The Zariski open set $\Hom_{\mathbf{g}}^{0}(\Ga,G)$ in $\cP$ where $G$ acts freely and properly is the space of $\chi$ whose image in $G$ does not lie in a group of the form $P\times Z(G)$, for $P$ a parabolic subgroup of $[G,G]=G$.  Such a restriction leads to the construction of the  \emph{smooth locus} 
\[
\cK_{\CC}=\Hom_{\mathbf{g}}^{0}(\Ga,G)/PG
\]
of the character variety $\Hom_{\mathbf{g}}(\Ga,G)/PG$. Since $G$ is simple, if nonempty, $\cK_{\CC}$ is a complex-analytic manifold of dimension 
\[
2(g-1)\dim_{\CC} G  +\sum_{i=1}^{n+m}\dim_{\CC}\mathcal{O}^{i},
\] 
where $\dim_{\CC} \OO_{i} =\textrm{rank} (\Ad (g_{i})-I)$ on $\mathfrak{g}$.
The infinitesimal deformations of a given $\chi\in \Hom_{\mathbf{g}}^{0}(\Ga,G)$, leaving fixed the conjugacy classes of its elliptic and parabolic generators, are determined by the space of \emph{parabolic 1-cocycles} $Z^{1}_{P}\left(\Ga,\mathfrak{g}_{\Ad\chi}\right)$, consisting of maps $z:\Ga\to \mathfrak{g}$ satisfying 
\[
z(\gamma_{1}\gamma_{2})=z(\gamma_{1})+\Ad\chi(\gamma_{1})\cdot z(\gamma_{2}),\quad \forall\;\gamma_{1},\gamma_{2}\in\Ga
\]
and 
\[
z(S_{i})\in(\Ad\chi(S_{i})-I)\mathfrak{g},\qquad z(T_{j})\in(\Ad\chi(T_{j})-I)\mathfrak{g}.
\]
The infinitesimal deformations tangential to the $PG$-orbit of $\chi$ in $\Hom_{\mathbf{g}}^{0}(\Ga,G)$ are precisely the 1-coboundary maps $z(\gamma)=(\Ad\chi(\gamma)-I)X$, $X\in\mathfrak{g}$, and hence the tangent space at the equivalence class $\{\chi\}\in\cK_{\CC}$ is modeled by the \emph{first parabolic cohomology space} $H^{1}_{P}(\Ga,\mathfrak{g}_{\Ad\chi})$ (see \cite{Shimura71} for details on these spaces; their dimensions were calculated by Weil \cite{Weil64}).   

\begin{remark}
In general, a representation $\chi\in \Hom^{0}_{\mathbf{g}}(\Ga,G)$ need not be irreducible. The collection of equivalence classes of irreducible representations (without inner automorphisms in $PG$) forms an open set in $\cK_{\CC}$. 
\end{remark}

Repeating the previous construction with $U$ instead, we conclude that there is a real-analytic manifold $\cK^{s}$ of real dimension $\dim_{\CC}\cK_{\CC}$, consisting of orbits of irreducible representations $\rho:\Ga\to U$ with prescribed conjugacy classes of parabolic and elliptic generators (where $PU$ acts freely and without adjoint invariants in $\mathfrak{u}$), since it readily follows that an irreducible representation $\rho: \Gamma \to U$ cannot have inner automorphisms in $PU$. In such case, the possible values of parabolic and elliptic elements are parametrized by the generalized flag manifolds $\cF^{i}=U/Z\left(g_{i}\right)$, together with an inclusion $\cK^{s}\hookrightarrow \cK_{\CC}$. In turn, the tangent space at an equivalence class $\{\rho\}\in\cK^{s}$ is modeled by the space $H^{1}_{P}(\Ga,\mathfrak{u}_{\Ad\rho})$.

\section{Deformation theory of $(\Ga,G)$-principal bundles on $\HH$}\label{section:deformation}

Given a representation $\chi:\Ga\to G$, \emph{the induced $(\Ga,G)$-bundle on $\HH$}  is the trivial bundle $\HH\times G$ with $\chi$ acting as bundle automorphisms commuting with the right $G$-action \cite{BS15,Seshadri69}. Hence, two $(\Ga,G)$-bundles induced by $\chi_{1}$, $\chi_{2}$ are \emph{equivalent} if there exists a holomorphic map $f:\HH\to G$, satisfying $f(\gamma\tau)=\chi_{1}(\gamma)f(\tau)\chi_{2}(\gamma)^{-1}$ $\forall \gamma\in\Ga$, $\tau\in\HH$, and whose limit exists at every cusp.\footnote{Observe that, in particular, the conjugacy classes of $\chi_{1}$ and $\chi_{2}$ at every parabolic and elliptic generator are necessarily the same.} In particular, the harmonic theory implies that a $(\Ga,G)$-bundle induced by $\chi$ is isomorphic to at most one induced by a unitary representation $\rho:\Ga\to U$, up to conjugation in $U$. Let $\{P_{\rho}\}$ denote the class of $(\Ga,G)$-bundles on $\HH$ induced by a class of irreducible representations $\{\rho\}\in\cK^{s}$. 
Our aim is to introduce complex coordinates and a manifold structure over the set $\cN^{s}=\{\{P_{\rho}\}_{\{\rho\}\in\cK}\}$. The results of the present section are a generalization of the analysis developed in \cite{TZ07} for stable parabolic bundles to $(\Ga,G)$-bundles on $\HH$, and are, in turn, inspired by the theory of quasiconformal mappings of Ahlfors-Bers. Thanks to \cite{BS15}, the resulting complex manifold corresponds to the moduli space of stable parahoric $G$-torsors with prescribed parabolic structure whose automorphism group is equal to $Z(G)$.

On the analytic side, by considering open sets in $\HH$ invariant under $\Ga$, one can define the sheaf $\mathbf{P}_{\rho}$ on $\HH$, of holomorphic $G$-valued maps which are automorphic with respect to $\rho$, and are bounded at the cusps. By the general deformation theory of Kodaira-Spencer, its infinitesimal deformations are encoded in the cohomology 
\[
H^{1}(\HH,\mathbf{E}_{\Ad\rho}),
\]
where $\mathbf{E}_{\Ad\rho}$ is the sheaf of holomorphic $\mathfrak{g}$-valued maps on $\HH$ which are automorphic with respect to the composition $\Ad\rho=\Ad\circ \rho$ and bounded at the cusps.
Such cohomology is naturally isomorphic, by Serre duality, to the dual space $\mathfrak{S}_{2}(\Ga,\Ad\rho)^{\vee}$, and  models the holomorphic cotangent space at $\{P_{\rho}\}$ (cf. \cite{Men14}, where the proof for the unitary case is sketched in detail; the argument for a general simple group $G$ is essentially the same, at least when the parabolic and elliptic weights are generic, and there is an underlying principal bundle with parabolic structure over $S=\Ga\setminus\HH^{*}$).\footnote{The geometric interpretation of $\mathfrak{g}$-valued cusp forms of weight 2 as global holomorphic sections of the sheaf $\mathbf{E}_{\Ad\rho}^{\vee}\otimes K_{S}$ enables the computation of the dimension of the space $\mathfrak{S}_{2}(\Ga,\Ad\rho)$ by means of the Riemann-Roch theorem (cf. \cite{Men14}).}

We have pointed out that the space $\mathfrak{S}_{2}(\Ga,\Ad\rho)$ of cusp forms of weight 2 corresponds to the kernel of the Cauchy-Riemann operator over the Hilbert space $\cH^{1,0}(\Ga,\Ad\rho)$. After consideration of duality and the Petersson inner product, such fact is consistent with the Dolbeaut isomorphism of sheaves 
\[
H^{1}(\HH,\mathbf{E}_{\Ad\rho})\cong H^{0,1}_{\delb}(\HH,\mathbf{E}_{\Ad\rho}).
\]
Moreover, thanks to the Petersson inner product, the previous isomorphism can be enhanced if one considers the Hodge Laplacian $\Delta=\delb^{*}$ over $\cH^{1,0}(\Ga,\Ad\rho)$; there is an additional isomorphism of the previous spaces with the space of harmonic $(0,1)$-forms with values in $\mathfrak{g}$, which are $\Ad\rho$-invariant and vanish at the cusps.

Since the family of $(\Ga,G)$-bundles on $\HH$ that we are considering is parametrized by the space $\cK^{s}$, for every $\rho$, there is an induced Kodaira-Spencer map
\[
H^{1}_{P}(\Ga,\mathfrak{u}_{\Ad\rho})\to H^1(\HH,\mathbf{E}_{\Ad\rho}),
\]
and the former space is, by the classical theorem of Shimura, isomorphic to the space $\fS_{2}(\Ga,\Ad\rho)$. Note that, as stated, Shimura's isomorphism is a real isomorphism. 

The following diagram summarizes and illustrates the role that the cusp forms of weight 2 and Shimura's isomorphism play in this dual set-up:

\vspace{8mm}

\centerline{
\begin{xy}
(0,0)*+{H^{1}_{P}\left(\Ga,\mathfrak{u}_{\Ad\rho}\right)}="a";%
(50,0)*+{H^{1}\left(\HH,\mathbf{E}_{\Ad\rho}\right)}="b"; 
(50,20)*+{\fS_{2}\left(\Ga,\Ad\rho\right)}="c";
(100,0)*+{H^{0,1}_{\delb}(\HH,\mathbf{E}_{\Ad\rho})}="e";
{\ar@{<->}_{\text{Kodaira-Spencer}} "a";"b"}
{\ar@{<->}^{\text{Shimura$\quad$}} "a";"c"}%
{\ar@{<->}_{\text{\v{C}ech+Serre}} "b";"c"}
{\ar@{<->}_{\text{Dolbeaut}} "b";"e"}
{\ar@{<->}^{\text{$\quad$Hodge+Serre}} "c";"e"}
\end{xy}}

\vspace{8mm}

\begin{remark}\label{remark:roots}
Let $R = \{\alpha\}$ be the set of roots associated to the abelian Lie algebra $\mathfrak {t}$, and $R = R^{+}\sqcup R^{-}$ a choice of positive and negative roots, which we will fix once and for all. The choice of positive roots $R^{+}$ determines the so-called \emph{Weyl alcove}, defined as 
\[
\cA(\mathfrak{t}) = \{t\in \mathfrak{t} \; :\; 0\leq \langle\alpha,t\rangle \leq 1\quad \forall \alpha\in R^{+}\}
\]
The Weyl alcove satisfies $\exp\left(2\pi\cA(\mathfrak{t})\right) = T$. Therefore, for each $1 \leq i \leq m + n$, there exists an element 
$t_{i}\in \cA(\mathfrak{t})$ in $\cA(\mathfrak{t})$, which will be called a choice of \emph{parabolic and elliptic weights} in $\mathfrak{u}$, satisfying $\exp(2\pi t_{i}) = g_{i}$.

The generalized flag manifolds $\cF^{i}$ introduced in section \ref{section:char} possess an induced complex structure in terms of the model $\cF^{i} \cong G/P_{i}$, where $P_{i} \subset G$ is the parabolic group with Lie algebra 
\[
\mathfrak{p_{i}} = \mathfrak{z}_{\mathfrak{g}}(t_{i}) \oplus\left(\displaystyle\bigoplus_{\alpha\in R^{+}_{i}} \mathfrak{g}_{\alpha}\right) =  \mathfrak{z}_{\mathfrak{g}}(t_{i}) \oplus \mathfrak{n}_{i}.
\] 
and $R_{i}^{+}\subset R^{+}$ is the subset of positive roots $\alpha$ for which $\alpha(t_{i}) > 0$. In turn, the centralizer $\mathfrak{z}_{\mathfrak{g}}(t_{i})$ coincides with the Lie algebra 
\[
\mathfrak{t}_{\CC}\oplus \left(\bigoplus_{\alpha \in R^{0}_{i}}\mathfrak{g}_{\alpha}\right)
\]
where $R^{0}_{i}\subset R$ is the subset of roots for which $\alpha(t_{i}) = 0$ and corresponds to the Levi complement Lie algebra of the unipotent radical Lie algebra $\mathfrak{n}_{i}$.

For every choice of parabolic generator $T_{i}$, let $U_{i}\in G$ be such that $\rho\left(T_{i}\right) = \Ad\left(U_{i}\right)\cdot g_{i}$. For the parabolic weight $t_{i}$, let us define the auxiliary holomorphic map $F_{i}:\HH \to G$ as 
\[
F_{i}(\tau) = U_{i}\exp \left(2\pi t_{i}\sigma^{-1}_{i}(\tau)\right)
\]
which satisfies the automorphic relation 
\begin{equation}
F_{i}\left(T_{i}\tau\right) = \rho\left(T_{i}\right)F_{i}(\tau).
\end{equation}
For any cusp form $\Phi\in\mathfrak{S}_{2}(\Gamma,\Ad\rho)$, the induced map $\Ad\left(F_{i}^{-1}\right)\cdot \Phi$ is $\langle T_{i}\rangle$-invariant, and consequently, there exist a holomorphic Fourier series expansion
\begin{equation}\label{eq:Fourier}
\Ad\left(F_{i}^{-1}(\sigma_{i}(\tau))\right)\cdot \Phi(\sigma_{i}(\tau))\sigma_{i}'(\tau) = \sum_{k = 0}^{\infty} B_{i}(k)q^{k},\qquad B_{i}(k)\in \mathfrak{g}.
\end{equation}
It follows that the cusp form condition \eqref{limit} is equivalent to 
\[
B_{i}(0) \in \mathfrak{n}_{i},
\]
a condition which is seen to be independent of the choice of representative $U_{i}$ in the corresponding $Z_{U}(g_{i})$-torsor. In fact, similar asymptotic formulas hold in the neighborhood of every elliptic fixed point, and are relevant in the study of the explicit deformations of parabolic and elliptic generators within the conjugacy classes $\mathcal{O}_{i}$. However, since they are not important in the determination of the cusp form condition, we will not consider them here.
\end{remark}

\begin{lemma}\label{lemma:fundamental} Let $\{\rho\}\in\cK^{s}$. 
\begin{itemize}
\item[(i)] For every $\Phi\in\mathfrak{S}_{2}(\Ga,\Ad\rho)$, there is a function $f_{\Phi}:\HH\to G$, holomorphic  and regular at each cusp, solving the differential equation $f_{\tau}f^{-1}=\Phi$, and satisfying
\begin{equation}\label{eq:monodromy1}
f_{\Phi}(\gamma\tau)=\rho(\gamma)f_{\Phi}(\tau)\chi(\gamma)^{-1}\qquad \forall\;\gamma\in \Ga,\;\; \tau\in\HH,
\end{equation}
for an irreducible representation $\chi \in \Hom^{0}_{\mathbf{g}}(\Gamma, G)$. More specifically, for every $m + 1 \leq i \leq m + n$, let $U_{i}\in U$ and $G_{i} \in G$ satisfy 
\[
\rho\left(T_{i}\right) = \Ad(U_{i})\cdot g_{i},\qquad \chi\left(T_{i}\right) = \Ad(G_{i})\cdot g_{i}. 
\]
Then
\[
f_{\Phi}(\sigma_{i}\tau) = U_{i}\left(\Ad\left(\emph{exp}(2\pi t_{i}\tau)\right)\cdot f^{i}_{\Phi}(\sigma_{i}\tau)\right)G_{i}^{-1},
\]
where $f^{i}_{\Phi}(\tau):\HH \to G$ is holomorphic and $\langle T_{i}\rangle$-invariant, and satisfies  
\begin{equation}\label{eq:parabolic1}
\lim_{\tau\to \tau_{i}} f^{i}_{\Phi}(\tau) = f_{i}\in P_{i}
\end{equation}
The set $\{f_{\Phi}\}$ of all such solutions is a torsor for the right action of  $G$. 
\item[(ii)] Conversely, given any function $f: \HH \to G$ as above, satisfying \eqref{eq:monodromy1}--\eqref{eq:parabolic1}, the function $\Phi = f_{\tau} f^{-1}$ belongs to $\mathfrak{S}_{2}(\Gamma, \Ad\rho)$, and is obviously invariant under the right $G$-action in $f$.
\end{itemize}
\end{lemma}
\begin{proof}
The existence of holomorphic solutions $f_{\Phi}:\HH\to G$ of the differential equation $f_{\tau}f^{-1}=\Phi$, regular at the cusps, is standard \cite{Malgrange84}. Because of the holomorphicity of $f_{\Phi}$ and the $\Ad\rho$-automorphicity of $\Phi$, a simple computation shows that the functions 
\[
\chi(\gamma,\tau)=f_{\Phi}(\gamma\tau)^{-1}\rho(\gamma)f_{\Phi}(\tau),
\]
are constant in $\tau$, and satisfy $\chi(\gamma_{1}\gamma_{2})=\chi(\gamma_{1})\chi(\gamma_{2})$. Hence
\[
f_{\Phi}(\gamma\tau)=\rho(\gamma)f_{\Phi}(\tau)\chi(\gamma)^{-1}, \qquad \forall\; \gamma\in\Ga,\;\; \tau\in\HH,
\]
for some $\chi:\Ga\to G$. Moreover, we have that $\chi(S_{i})\in\mathcal{O}_{i}$ and $\chi(T_{i})\in \mathcal{O}_{i}$, since it readily follows from the holomorphicity and regularity of $f_{\Phi}$ that
\[
\chi(S_{i})= f_{\Phi}(e_{i})^{-1}\rho(S_{i})f_{\Phi}(e_{i}),\quad  \chi(T_{j})= f_{\Phi}(p_{j})^{-1}\rho(T_{j})f_{\Phi}(p_{j}).
\]
Since $f_{\Phi}$ is holomorphic and regular at the cusps, it determines an isomorphism of sheaves $\mathbf{P}_{\chi}\cong\mathbf{P}_{\rho}$, implying the irreducibility of $\chi$.
Moreover, the group elements $f_{i}$ in \eqref{eq:parabolic1} are related to the constant Fourier coefficients $B_{i}(0) \in \mathfrak{n}_{i}$ in \eqref{eq:Fourier} as
\begin{equation}\label{eq:parabolic2}
B_{i}(0) = \left(I - \Ad(f_{i})\right)\cdot t_{i},
\end{equation}
and consequently, equation \eqref{eq:parabolic2} is nothing but a reformulation of the cusp form condition at each cusp $p_{i}$ in terms of the asymptotic behavior of the map $f_{\Phi}$. Therefore, it follows that $f_{i}\in P_{i}$.  
Finally, uniqueness of solutions is granted provided that a given value $f(\tau_{0})$ is fixed at some $\tau_{0}\in\HH$, and hence $f_{\Phi}$ is defined up to right multiplication by an element in $G$.  

The converse statement is immediate, in view of the reformulation of the cusp form condition in \eqref{eq:parabolic2}.
\end{proof}

\begin{remark} The right $G$-action on $f_{\Phi}$ induces an effective $PG=G/Z(G)$-action in $\chi$ by conjugation, determining a full class $\{\chi\}$ of equivalent representations.
Each function $f_{\Phi}$ induces a holomorphic equivalence of $(\Ga,G)$-principal bundles
which is invariant if $\rho$ (and consequently $\Phi$) is replaced by an equivalent representation, under conjugation in $U$. The residual right multiplication of $f_{\Phi}$
by elements in $Z(G)$ corresponds to the holomorphic automorphisms of $P_{\rho}$. Thus, lemma \ref{lemma:fundamental} determines a correspondence between the elements in $\mathfrak{S}_{2}(\Ga,\Ad\rho)$ and the equivalence classes of irreducible representations $\{\chi\}$ whose induced $(\Gamma,G)$-bundles satisfy  $P_{\chi} \cong P_{\rho}$.\footnote{\textbf{Note added in proof:} A result analogous to lemma \ref{lemma:fundamental} first appeared in the work of Florentino \cite{Flo01} (lemma 2), in connection to the study of the existence problem for the so-called Schottky uniformizations of vector bundles on Riemann surfaces. Although his result is described in the case of $G = \GL(r,\CC)$, and purely hyperbolic Fuchsian groups, the principle and ideas that make it work are essentially the same than the ones presented here. In particular, such result is used in \cite{Flo01} to characterize the set of all irreducible representations $\chi:\Gamma \to \GL(r,\CC)$ that induce a given equivalence class $\{E\to \Sigma\}$ of stable vector bundles on a compact Riemann surface $\Sigma$ uniformized by $\Gamma$ (lemma 3).}
\end{remark}

\begin{corollary}\label{cor:F}
For every pair $\left\{P_{\rho},\Phi\right\}$, $\Phi\in\mathfrak{S}_{2}(\Ga,\Ad\rho)$ (up to conjugation in $U$), there is a unique induced irreducible representation $\chi:\Ga\to G$ as in lemma \ref{lemma:fundamental2}, up to conjugation in $G$, such that $P_{\chi} \cong P_{\rho}$. Conversely, for any holomorphic map $f:\HH\to G$ satisfying \eqref{eq:monodromy1} and \eqref{eq:parabolic1} for some irreducible $\chi\in \Hom_{\mathbf{g}}^{0}(\Gamma, G)$, regular at each cusp, and inducing an isomorphism $P_{\chi}\cong P_{\rho}$, it follows that $f_{\tau}f^{-1}\in \mathfrak{S}_{2}(\Ga,\Ad\rho)$.
The map $\cF:T^{*}\cN^{s}\to \cK_{\CC}$, $\left\{P_{\rho},\Phi\right\}\mapsto \{\chi\}$ is 1-to-1.
\end{corollary}

\begin{proof}
Each map $f_{\Phi}$ is essentially an isomorphism of $(\Gamma, G)$-bundles for $P_{\rho}$ and $P_{\chi}$. 
If for $ \{P_{\rho},\Phi\}$,  $\{P_{\rho'},\Phi'\}$, we have that $\{\chi\}=\{\chi'\}$,  then after conjugation in $G$, $f_{\Phi'}(f_{\Phi})^{-1}$ is a regular $G$-valued $\rho'\otimes\rho^{\vee}$-automorphic function, hence a constant lying in $U$. Hence $\rho\cong \rho'$, or $\{\rho\}=\{\rho'\}$, which implies that $\Phi=\Phi'$ up to conjugation in $U$, or $\{\Phi\}=\{\Phi'\}$. 
\end{proof}

Let $\cN^{s}$ be the set of equivalence classes of $(\Ga, G)$-bundles on $\HH$ of the form $P_{\rho}$, for $\{\rho\}\in\cK^{s}$.  We wish to turn $\cN^{s}$ into a complex manifold, in such a way that the holomorphic tangent space at a point $\{P_{\rho}\}\in \cN^{s}$ is modeled by the spaces $\overline{\fS_{2}(\Ga,\Ad\rho)}\cong \fS_{2}(\Ga,\Ad\rho)^{\vee}$, for any representative $\rho$ (observe that such spaces are isomorphic under conjugation). In order to introduce complex coordinates in $\cN^{s}$, we need to provide a method of construction of local deformations of Cauchy-Riemann operators. The fundamental result providing a complex-analytic description of the deformation theory of $(\Ga,G)$-bundles is the following (cf. \cite{Malgrange84}).

\begin{lemma}\label{lemma:fundamental2} 
Let $\rho: \Ga\to U$ be an irreducible representation in $\cK^{s}$.  For every $\nu\in\overline{\mathfrak{S}_{2}(\Ga,\Ad\rho)}$ and $\epsilon\in\CC$ sufficiently close to 0, there is a function $f^{\epsilon\nu}:\HH\to G$ solving 
the differential equation $f^{-1}f_{\bar{\tau}}=\epsilon\nu$, 
satisfying
\begin{equation}\label{eq:monodromy2}
f^{\epsilon\nu}(\gamma\tau)=\rho^{\epsilon\nu}(\gamma)f^{\epsilon\nu}(\tau)\rho(\gamma)^{-1}
\end{equation}
for some irreducible representation $\rho^{\epsilon\nu}:\Ga\to U$ in $\cK^{s}$, real-analytic and regular at each cusp.
The set $\{f^{\epsilon\nu}\}$ of such solutions is a torsor for the left action of  $U$.
moreover, the conjugacy classes of parabolic and elliptic elements are preserved.
\end{lemma}
\begin{proof}
This is essentially the same argument as in \cite{TZ07}, p.124: existence of antiholomorphic solutions $f:\HH\to G$, regular at the cusps follows from lemma \ref{lemma:fundamental}; then the equation $f^{-1}f_{\bar{\tau}}=\epsilon\nu$ admits an antiholomorphic solution $f_{-}$, depending holomorphically in $\epsilon$. 
Consider any such $f_{-}$, and for any $\gamma\in\Ga$, consider the function $\chi_{\gamma}(\overline{\tau})=\left(f_{-}^{-1}\circ \gamma\right)\rho(\gamma)f_{-}$. Since $\nu$ is $\Ad\rho$-automorphic,
\[
\frac{\partial}{\partial\overline{\tau}}\chi_{\gamma}(\overline{\tau})=\epsilon\left(f_{-}^{-1}\circ\gamma\right)\rho(\gamma)\nu f_{-}-\epsilon \left(f_{-}^{-1}\circ\gamma\right)(\nu\circ\gamma)\overline{\gamma'}\rho(\gamma)f_{-}\equiv 0.
\]
Since $\chi_{\gamma}$ is antiholomorphic, it must be equal to a constant $\chi(\gamma)$, and clearly $\chi(\gamma_{1}\gamma_{2})=\chi(\gamma_{1})\chi(\gamma_{2})$. The function $f_{-}$ satisfies the functional equation $f_{-}\circ\gamma =\chi(\gamma)f_{-}\rho(\gamma)^{-1}$ $\forall \gamma\in\Ga, \tau\in\HH$, for some representation $\chi:\Ga\to G$, defined up to conjugation.
Following the main result in \cite{BS15}, and by the openness of stability, for sufficiently small $\epsilon$, the monodromy representation associated to $f_{-}$ defines a $(\Ga,G)$-bundle on $\HH$ which is equivalent to an irreducible unitary one, and consequently, there exists a holomorphic map $f_{+}:\HH\to G$, regular at each cusp, such that  
\[
f_{+}f_{-}=f^{\epsilon\nu}:\HH\to G
\]
is a solution of \eqref{eq:monodromy2}, depending real-analytically in $\tau$ and $\epsilon$, and satisfying 
$f^{\epsilon\nu}(\gamma\tau)=\rho^{\epsilon\nu}(\gamma)f^{\epsilon\nu}(\tau)\rho(\gamma)^{-1}$ for all $\gamma\in\Gamma,$ with $\rho^{\epsilon\nu}:\Ga\to U$ irreducible. Clearly, such $f^{\epsilon\nu}$ is defined up to left multiplication by an element in $U$. 

Finally, as in the proof of lemma \ref{lemma:fundamental}, the regularity of $f^{\epsilon\nu}$ at the cusps and elliptic fixed points implies that the conjugacy classes of such group elements are preserved under deformations. This concludes the proof.
\end{proof}

\begin{remark}
Observe that the collection of solutions $\{f^{\epsilon\nu}\}$ in lemma \ref{lemma:fundamental2} is a torsor for $PU$, and is thus an invariant of the classes $\{\rho\}$, $\{\rho^{\epsilon\nu}\}$.  Such solutions define a map of $(\Ga,G)$-bundles, and by pull-back, allow to express the Cauchy-Riemann operator of $P_{\rho^{\epsilon\nu}}$ over $P_{\rho}$ as
\[
\left(f^{\epsilon\nu}\right)^{-1}\circ\delb\circ f^{\epsilon\nu}=\delb +\epsilon\nu.
\]
This is our set of complex coordinates in $\cN^{s}$, analogous to Bers' coordinates on Teichm\"uller spaces. It can be verified \cite{TZ07} that such coordinates transform holomorphically, and thus turn $\cN^{s}$ into a complex manifold. Its complex dimension, which depends on the fixed set of group elements $\mathbf{g}=\{g_{i}\}_{i=1}^{m+n}$, is equal to $(g-1)\dim U +\sum_{i=1}^{m+n}\dim_{\CC}\cF^{i}$.
\end{remark}

\begin{corollary}\label{cor:coord}
Complex coordinates, analogous to the Bers' coordinates on Teichm\"uller spaces, can be introduced in $\cN^{s}$.  Given  $f^{\epsilon\nu}$ as in lemma \ref{lemma:fundamental2},
the holomorphic structure corresponding to $\rho^{\epsilon\nu}$ is realized in the $\Ga$-vector bundle over $\HH$ associated to the sheaf $\mathbf{E}_{\Ad\rho}$, by means of the $\bar{\partial}$-operator  $\left(\Ad f^{\epsilon\nu}\right)^{-1}\circ \bar{\partial}\circ \Ad f^{\epsilon\nu}$.
\end{corollary}

\begin{lemma}\label{lemma:vanishing}
Let $\nu\in\overline{\mathfrak{S}_{2}(\Ga,\Ad\rho)}$ and $\epsilon\in\CC$ sufficiently small. Then
\begin{equation} \label{h-first}
\left.\frac{\partial}{\partial\varepsilon}\left( (f^{\varepsilon\nu})^{-1}\overline{f^{\varepsilon\nu}}\right)\right|_{\varepsilon=0}
=\left.\frac{\partial}{\partial\bar{\varepsilon}}\left( (f^{\varepsilon\nu})^{-1}\overline{f^{\varepsilon\nu}}\right)\right|_{\varepsilon=0}=0
\end{equation}
where $g\mapsto \overline{g}$ is the involution in $G$ having $U$ as its set of fixed points.  
\end{lemma}
\begin{proof}
Since $f^{\epsilon\nu}$ is defined as a product of a holomorphic and an antiholomorphic function, it readily follows that the functions 
\[
\left.\frac{\partial}{\partial\varepsilon}\left( (f^{\varepsilon\nu})^{-1}\overline{f^{\varepsilon\nu}}\right)\right|_{\varepsilon=0},\qquad 
\left.\frac{\partial}{\partial\bar{\varepsilon}}\left( (f^{\varepsilon\nu})^{-1}\overline{f^{\varepsilon\nu}}\right)\right|_{\varepsilon=0},
\]
are harmonic, and correspond to harmonic sections of the sheaf $\mathbf{E}_{\Ad\rho}$, thus equal to an element in the center of $\mathfrak{g}$, which is necessarily zero since $G$ is simple.
\end{proof}

Our previous construction determines a tautological real-analytic isomorphism $\cN^{s}\cong\cK^{s}$, and consequently, a real-analytic inclusion $\cN^{s}\hookrightarrow\cK_{\CC}$. 
As we will see, the complex-analytic mechanism that we have implemented to deform a given $(\Ga,G)$-principal bundle $P_{\rho}$ can be used to extend the previous isomorphism to a real-analytic monomorphism $\cF:T^{*}\cN^{s}\hookrightarrow \cK_{\CC}$. However, the map $\cF$ that we will construct does not correspond, in general, to the monodromy of (restricted) stable pairs \cite{Hitchin87a, Hitchin87b}. The simplicity of the construction of the extension of $\cF$ that we will provide is an indicator of its naturality: the fibers over a given $\{P_{\rho}\}\in \cN^{s}$ are put in 1-to-1 correspondence with the classes of representations $\{\chi\}\in\cK_{\CC}$ inducing an equivalent  $(\Ga,G)$-bundle over $\HH$.

Let us consider the holomorphic cotangent bundle $T^{*}\cN^{s}$, whose points would be denoted as $\{P_{\rho},\Phi\}$. It follows that the tangent space at a point $\{P_{\rho},0\}\in T^{*}\cN$ in the zero section is isomorphic to the direct sum
\[
H^{1}\left(\HH,\mathbf{E}_{\Ad\rho}\right)^{\vee}\oplus H^{1}\left(\HH,\mathbf{E}_{\Ad\rho}\right)\cong \mathfrak{S}_{2}(\Ga,\Ad\rho)\oplus\overline{\mathfrak{S}_{2}(\Ga,\Ad\rho)}.
\]



In the next section, we will show that the differential of the map $\cF$ is indeed a bijection at any point.

\section{Shimura's isomorphism}\label{section:Shimura}

\begin{definition}
An \emph{Eichler integral of weight 0 with the representation} $\Ad\rho$, for $\rho:\Ga\to U$, is a holomorphic function $\cE:\HH\to\mathfrak{g}$ satisfying 
\[
\cE(\gamma\tau)=\Ad\rho(\gamma)\cdot\cE(\tau)+z(\gamma),\quad \gamma\in\Ga,\,\tau\in\HH,
\]
where $z:\Ga\to\mathfrak{g}$ is an element in $Z_{P}^{1}(\Ga,\mathfrak{g}_{\Ad\rho})$, and moreover, such that $\cE'$ is a $\mathfrak{g}$-valued automorphic form of weight 2 for $\Ad\rho$, which is meromorphic at each parabolic fixed point of $\Gamma$. 
\end{definition}

It follows that for every cusp form $\Phi$ of weight 2 with a representation $\Ad\rho$, any choice $\cE_{\Phi}$ of antiderivative of it is an Eichler integral of weight 0,
\[
\cE_{\Phi}(\gamma\tau)=\Ad\rho(\gamma)\cdot \cE_{\Phi}(\tau)+z_{\Phi}(\gamma),\quad \gamma\in\Ga,\,\tau\in\HH.
\]
Since $z_{\Phi}:\Ga\to \mathfrak{g}$ is a $\mathfrak{g}$-vaued parabolic 1-cocycle with the representation $\Ad\rho$, then, in particular, $\re (z_{\Phi})\in Z^{1}_{P}(\Ga,\mathfrak{u}_{\Ad\rho})$. It can be verified that, if a different choice of antiderivative for $\Phi$ is made, the difference between the corresponding $\mathfrak{u}$-valued 1-cocycles is in fact a coboundary. Therefore, it is possible to associate a well-defined cohomology class in $H^{1}_{P}\left(\Gamma,\mathfrak{u}_{\Ad\rho}\right)$ to every cusp form of weight 2 with the representation $\Ad\rho$. \emph{Shimura's isomorphism} \cite{Shimura71} 
is the statement of the fact that the previous map is an isomorphism, that is, 
\[
\mathfrak{S}_{2}(\Ga,\Ad\rho)\cong H^{1}_{P}(\Ga,\mathfrak{u}_{\Ad\rho}), 
\]
where the isomorphism is given explicitly by 
\[
\cL_{S}(\Phi)=\text{the cohomology class of}\;\, \re (z_{\Phi}).
\]

Shimura's isomorphism is an isomorphism of real vector spaces. Since  $\mathfrak{g}\cong\mathfrak{u}_{\CC}$ implies the canonical identification
$H^{1}_{P}(\Ga,\mathfrak{g}_{\Ad\rho})\cong H^{1}_{P}(\Ga,\mathfrak{u}_{\Ad\rho})_{\CC}$, and recalling that for a complex vector space $V$, $(V_{\RR})_{\CC}\cong V\oplus\overline{V}$, it is also possible to formulate Shimura's isomorphism as an isomorphism of complex vector spaces, in the form
\[
\mathfrak{S}_{2}(\Ga,\Ad\rho)\oplus\overline{\mathfrak{S}_{2}(\Ga,\Ad\rho)}\cong H^{1}_{P}(\Ga,\mathfrak{g}_{\Ad\rho}).
\]
For fixed $\Phi\in\mathfrak{S}_{2}(\Ga,\Ad\rho)$ and $\nu\in\overline{\mathfrak{S}_{2}(\Ga,\Ad\rho)}$, let us denote 
\[
\dot{f}_{+}=\left.\frac{\partial f_{\epsilon\Phi}}{\partial \epsilon}\right|_{\epsilon=0},\qquad \dot{f}^{-}=\left.\frac{\partial f^{\epsilon\nu}}{\partial \bar{\epsilon}}\right|_{\epsilon=0}.
\]

\begin{lemma}\label{lemma:var}
For any $\Phi\in\mathfrak{S}_{2}(\Ga,\Ad\rho)$, $\nu\in\overline{\mathfrak{S}_{2}(\Ga,\Ad\rho)}$ and $\epsilon\in\CC$ sufficiently small, the functions  
$\dot{f}_{+}$ and $\dot{f}^{-}$ are Eichler integrals of weight 0, satisfying
\begin{equation} \label{var-formulas}
(\dot{f}_{+})'=\Phi,\qquad (\dot{f}^{-})'=\overline{\nu}.
\end{equation}
\end{lemma}
\begin{proof}
That $\dot{f}_{+}$ is an antiderivative of $\Phi$ is automatic from \eqref{eq:monodromy1} and the holomorphicity of $f_{\epsilon\Phi}$.
The argument for $\dot{f}^{-}$ is only slightly more elaborate: first of all, it follows from \eqref{eq:monodromy2} that $(\dot{f}^{-})_{\bar{\tau}}=0$, so $\dot{f}^{-}$ is holomorphic.
Now, it follows from lemma \ref{lemma:vanishing} that 
\[
\overline{\left.\frac{\partial f^{\epsilon\nu}}{\partial \epsilon}\right|_{\epsilon=0}}=\left.\frac{\partial f^{\epsilon\nu}}{\partial \bar{\epsilon}}\right|_{\epsilon=0}
\]
but the left-hand side is an antiderivative for $\overline{\nu}$ as a consequence of \eqref{eq:monodromy2}, and the claim follows.
\end{proof}

\begin{proof}[\textbf{Proof of theorem 1}]
We have seen in section \ref{section:deformation} that the tangent space at any point $\{P_{\rho},0\}\in T^{*}\cN^{s}$ is modeled by the space $\mathfrak{S}_{2}(\Ga,\Ad\rho)\oplus\overline{\mathfrak{S}_{2}(\Ga,\Ad\rho)}$. In turn, the tangent space at $\{\rho\}=\cF(\{P_{\rho},0\})\in \cK_{\CC}$ is $H^{1}_{P}(\Ga,\mathfrak{g}_{\Ad\rho})$.

The isomorphism $(V_{\RR})_{\CC}\cong V\oplus\overline{V}$ is determined by the $\sqrt{-1}$ and $-\sqrt{-1}$-eigenspaces of the map $(v,w)\mapsto (-w,v)$, $v,w\in V$, which are complex subspaces of $(V_{\RR})_{\CC}$ spanned, respectively, by elements of the form $(v,-\sqrt{-1}v)$ and $(v,\sqrt{-1}v)$. Under this identification, the isomorphism 
\[
\cL_{S}:\mathfrak{S}_{2}(\Ga,\Ad\rho)\oplus\overline{\mathfrak{S}_{2}(\Ga,\Ad\rho)}\to H^{1}_{P}(\Ga,\mathfrak{g}_{\Ad\rho})
\]
is given explicitly as 
\begin{equation}\label{eq:Shimura iso}
\cL_{S}(\Psi,\nu)= [z_{\Psi}+\overline{z_{\overline{\nu}}}],
\end{equation}
where $[\cdot]$ denotes the cohomology class of an element in $Z^{1}(\Ga,\mathfrak{g}_{\Ad\rho})$.


Choose an arbitrary representative $\rho$ in its equivalence class, and let $\epsilon$ be sufficiently small, and $\chi_{\epsilon\Psi}$ (resp. $\{\rho^{\epsilon\nu}\}$) the associated irreducible representation in $\cK_{\CC}$ (resp. $\cK^{s}$). Then, the maps $\{\epsilon\Psi\}\mapsto \{\chi_{\epsilon\Psi}\}$, $\{\epsilon\nu\}\mapsto \{\rho^{\epsilon\nu}\}$ defined in corollaries \ref{cor:F} and \ref{cor:coord} yield 
\[
d\mathscr{F}_{\{P_{\rho},0\}}(\Psi)=-\left.\frac{\partial}{\partial\epsilon}(\chi_{\epsilon\Psi})\right|_{\epsilon=0}\rho^{-1},\qquad d\mathscr{F}_{\{P_{\rho},0\}}(\nu)=\left.\frac{\partial}{\partial\epsilon}(\rho^{\epsilon\nu})\right|_{\epsilon=0}\rho^{-1},
\]
as $\mathfrak{g}$-valued parabolic 1-cocycles. 
On the other hand,  lemma \ref{lemma:var} states that $\dot{f}_{+}$ is an  Eichler integral of weight 0 with the representation $\Ad\rho$, satisfying $(\dot{f}_{+})'=\Psi$. Its associated parabolic 1-cocycle is seen to be 
\[
z_{+}(\gamma)=-\left.\frac{\partial}{\partial\epsilon}(\chi_{\epsilon\Psi}(\gamma))\right|_{\epsilon=0}\rho(\gamma)^{-1}.
\]
In turn, $\dot{f}^{-}$ is also an Eichler integral of weight 0, satisfying  $(\dot{f}^{-})'=\overline{\nu}$. From \eqref{eq:monodromy2}  and \eqref{h-first}, its associated parabolic 1-cocycle is 
\[
z^{-}(\gamma)=\left.\frac{\partial}{\partial\bar{\epsilon}}(\rho^{\epsilon\nu}(\gamma))\right|_{\epsilon=0}\rho(\gamma)^{-1}
=\overline{\left.\frac{\partial}{\partial\epsilon}(\rho^{\epsilon\nu}(\gamma))\right|_{\epsilon=0}\rho(\gamma)^{-1}},
\]
where the last equality follows from the algebraic identity $\overline{\rho^{\epsilon\nu}} = \rho^{\epsilon\nu}$ by direct computation.

Finally, since the conjugation of $\rho$ induces a trivial change in the parabolic 1-cocycles, leaving their cohomology class fixed, we conclude
\[
d\mathscr{F}_{\{P_{\rho},0\}}(\Psi,0)=\cL_{S}(\Psi,0),\qquad d\mathscr{F}_{\{P_{\rho},0\}}(0,\nu)=\cL_{S}(0,\nu),
\]
which completes the proof.
\end{proof} 

In particular, theorem \ref{main theo} ensures that the restriction of $\cF$ to the image of $s_{0}$ is a real-analytic diffeomorphism onto $\cK^{s}$.  Let $\pi :T^{*}\cN^{s}\to \cN^{s}$ be the cotangent bundle projection. In general, the vertical subbundle $V\to T^{*}\cN^{s}$ of $T\left(T^{*}\cN^{s}\right)$, defined as $\ker\left(d\pi\right)$, is canonically isomorphic to $\pi^{*}\left(T^{*}\cN^{s}\right)$, and there is an induced exact sequence of vector bundles over $T^{*}\cN^{s}$,
\[
0 \to \pi^{*}\left(T^{*}\cN^{s}\right)\to T\left(T^{*}\cN^{s}\right)\to \pi^{*}\left(T\cN^{s}\right) \to 0,
\]
even though for an arbitrary $\Phi\in \mathfrak{S}_{2}(\Ga,\Ad\rho)$, there is no canonical splitting of $T_{\{P_{\rho},\Phi\}}\left(T^{*}\cN^{s}\right)$, and consequently no natural generalization of the result in the complement of $s_{0}$ on $T^{*}\cN^{s}$.  Nevertheless, it readily follows that, indeed, the differential of $\cF$ is a bijection over any point $\{P_{\rho},\Phi\}\in T^{*}\cN^{s}$. To see this, consider the map $\mathrm{pr}:\cF(T^{*}\cN^{s})\subset \cK_{\CC}\to \cK^{s}$, defined by requiring
\[
\cF_{U}\circ \pi =\mathrm{pr}\circ \cF,
\]
where $\cF_{U}=\cF\circ s_{0}:\cN^{s}\to\cK^{s}$. 
We know that $d\cF_{U}$ is an isomorphism at any point in $\cN^{s}$, and 
consequently, $d\left(\cF_{U}\circ \pi \right)$ is surjective at any $\{P_{\rho},\Phi\}\in T^{*}\cN^{s}$, and therefore, the same is true for $d\left(\mathrm{pr}\circ \cF\right)$. From this, it follows that the kernel of $d\cF$ at a point $\{P_{\rho},\Phi\}\in T^{*}\cN^{s}$ must be a subspace of $V_{\{P_{\rho},\Phi\}}=\ker \left(d\pi_{\{P_{\rho},\Phi\}}\right)$. On the other hand, the isomorphism $V\cong \pi^{*}\left(T^{*}\cN^{s}\right)$ implies that $d\cF |_{V}$ is given at any point $\{P_{\rho},\Phi\}$ simply as 
\[
\Psi \mapsto \cL_{S}(\Psi,0). 
\]
Hence $d\cF |_{V_{\{P_{\rho},\Phi\}}}$ has maximal rank, and we have concluded the following corollary.


\begin{corollary}
The map 
\[
\cF:T^{*}\cN^{s}\to\cK_{\CC},
\] 
is a  real-analytic monomorphism. Its restriction to the zero section is in correspondence with the inclusion $\iota:\cK^{s}\hookrightarrow \cK_{\CC}$. The image $\cF\left(T^{*}\cN^{s}\right)=\cK_{\CC}^{s}$ is the Zariski open set in $\cK_{\CC}$ of equivalence classes of $\chi$ whose associated $(\Ga,G)$-bundle is isomorphic to an irreducible unitary one, without nontrivial automorphisms.
There is a foliation of $\cK_{\CC}^{s}$, parametrized by $\cK^{s}$, whose leaves are given by the images of each $T^{*}_{\{P_{\rho}\}}\cN^{s}\cong\mathfrak{S}_{2}(\Ga,\Ad\rho)$ under $\cF$.
\end{corollary}

\begin{remark}
Altogether, the real isomorphisms $\mathfrak{S}_{2}(\Ga,\Ad\rho)\cong H^{1}_{P}(\Ga,\mathfrak{u}_{\Ad\rho})$ turn the smooth locus of the real character variety $\cK$ into an almost complex manifold. Theorem \ref{main theo}, restricted to $\cN^{s}$, shows indirectly that this almost complex structure on  $\cK^{s}$ would be integrable (in the same spirit of \cite{NS64}), thus determining a complex structure over $\cK^{s}$. 
\end{remark}

\begin{remark}
The Petersson inner product in each $\overline{\mathfrak{S}_{2}(\Ga,\Ad\rho)}$ induces a K\"ahler stucture in $\cN^{s}$, which in terms of holomorphic normal coordinates, is given as
\[
\Omega\left(\frac{\partial}{\partial\epsilon(\nu_{1})},\frac{\partial}{\overline{\partial\epsilon(\nu_{2})}}\right)=\frac{\sqrt{-1}}{2}\langle\nu_{1},\nu_{2}\rangle_{P}.
\]
Its imaginary part is a closed, real 2-form defining a symplectic structure in $\cN^{s}$. 

Since cup product in parabolic cohomology corresponds to wedge product (defined in terms of the Killing form) in the space of cohomology classes of $\Ad\rho$-automorphic 1-forms on $\HH$ with values in $\mathfrak{u}$ and compactly supported within a fundamental region of $\Ga$, a space in turn isomorphic to $H^{1}(\HH,\mathbf{E}_{\Ad\rho})$ (cf. \cite{NS64}, Proposition 4.4), Shimura's isomorphism is an isomorphism of symplectic vector spaces, for the symplectic form on $H^{1}_{P}(\Ga,\mathfrak{u}_{\Ad\rho})$ given as
\[
H^{1}_{P}(\Ga,\mathfrak{u}_{\Ad\rho})\times H^{1}_{P}(\Ga,\mathfrak{u}_{\Ad\rho})\to H^{2}(\Ga,\RR)\cong\RR,
\]
\[
(u_{1},u_{2})\mapsto (u_{1}\smile u_{2})(\gamma_{1},\gamma_{2})=-K\left(u_{1}(\gamma_{1}),\Ad\rho(\gamma_{1})u_{2}(\gamma_{2})\right).
\]
The latter defines Goldman's symplectic structure in $\cK^{s}$ \cite{Gol84}.
Thus, the map $\cF_{U}$ is also a symplectomorphism between $\cN^{s}$ and $\cK^{s}$. 
\end{remark}

\bibliographystyle{amsalpha}

\providecommand{\bysame}{\leavevmode\hbox to3em{\hrulefill}\thinspace}
\providecommand{\MR}{\relax\ifhmode\unskip\space\fi MR }
\providecommand{\MRhref}[2]{%
  \href{http://www.ams.org/mathscinet-getitem?mr=#1}{#2}
}
\providecommand{\href}[2]{#2}

\end{document}